\documentclass{article}

\usepackage{amsmath,amsthm,amssymb}

\def\N{\mathbb N}

\def\A{\mathcal A}
\def\Codec{\mathrm{Dec}}
\def\Z{\mathcal Z}

\def\P{\mathcal P}
\def\u{\mathbf u}
\def\Pred{\mathcal{P}^{\mathrm{rel}}}
\def\AC{\mathrm{AC}}

\newtheorem{lemma}{Lemma}[section]
\newtheorem{proposition}[lemma]{Proposition}

\newtheorem{corollary}[lemma]{Corollary}
\newtheorem{observation}[lemma]{Observation}

\theoremstyle{definition}
\newtheorem{definition}[lemma]{Definition}

\theoremstyle{remark}
\newtheorem{remark}[lemma]{Remark}

\begin{document}

\title{Abelian complexity and Abelian co-decomposition}

\author{Ond\v rej Turek
\medskip \\
\normalsize{Laboratory of Physics, Kochi University of Technology} \\
\normalsize{Tosa Yamada, Kochi 782-8502, Japan} \\
\normalsize{email: \texttt{ondrej.turek@kochi-tech.ac.jp}}
}


\maketitle



\begin{abstract}
We propose a technique for exploring the abelian complexity of recurrent infinite words, focusing particularly on infinite words associated with Parry numbers. Using that technique, we give the affirmative answer to the open question posed by Richomme, Saari and Zamboni, whether the abelian complexity of the Tribonacci word attains each value in $\{4,5,6\}$ 
infinitely many times.
\end{abstract}

\section{Introduction}

Abelian complexity is now a widely studied property of infinite words. The first appearance of the idea dates back to the seventies, when Coven and Hedlund realized that periodic words and Sturmian words can be alternatively characterized using Parikh vectors~\cite{CoHe}.
Their results have been recently generalized by Richomme, Saari and Zamboni in~\cite{RSZ}, where the term ``abelian complexity'' itself has been introduced. That work initiated a systematic study of abelian properties of words: \cite{RSZ} was quickly followed by a series of related papers, both on general topics in abelian complexity~\cite{CR,CRSZ} and on abelian complexity of concrete infinite words~\cite{RSZ2} as well as of certain families of words~\cite{Tu2,BBT}.

In general, calculating $\AC(n)$ for a given infinite word is a difficult problem. Only a few results, as well as effective methods, are known so far. Even the simpler question, whether a given value $k$ of the function $\AC$ is attained finitely or infinitely many times, is usually hard to be answered, expecially for $k$ different from the extremal values $\max\AC$, $\min\AC$. For example, it is known that the abelian complexity of the Tribonacci word $\mathbf{t}$ (recall that $\mathbf{t}$ is the fixed point of the substitution $0\mapsto01$, $1\mapsto02$, $2\mapsto0$) satisfies $\AC_\mathbf{t}(n)\in\{3,4,5,6,7\}$ for all $n$, but only for the values $3$ and $7$ it is proved that they are attained infinitely many times, see~\cite{RSZ2}. Similarly, for $\u^{(p)}$ being the fixed point of the substitution $L\mapsto L^pS$, $S\mapsto M$, $M\mapsto L^{p-1}S$ for an arbitrary $p\geq2$, it has been proved $\AC_{\u^{(p)}}(n)\in\{3,4,5,6,7\}$, but so far only the value $7$ is known to be attained infinitely many times~\cite{Tu2} (for additional information on those words see~\cite{FrGaKr,KP}).

In this paper we develop a method for dealing with the abelian complexity of recurrent infinite words, which is fitted especially to infinite words associated with Parry numbers. It can be used for an effective calculation of $\AC(n)$ for a given $n$, as well as for proving that a certain value of $\AC$ is attained infinitely many times. To demonstrate that, we consider an open question posed by Richomme, Saari and Zamboni in~\cite{RSZ2}, whether the abelian complexity of the Tribonacci word attains each value in $\{4,5,6\}$ infinitely often; with the help of the method, we obtain the affirmative answer.

\section{Preliminaries}

An alphabet $\A$ is a finite set of symbols called \emph{letters}.
Any concatenation of letters from $\A$ is called a \emph{word}. The set $\A^*$ of all finite words over $\A$ including the empty word $\varepsilon$ is a free monoid. For any $w=w_0w_1w_2\cdots w_{n-1}\in\A^*$, the \emph{length} of $w$ is defined as $|w|=n$. The length of the empty word is by definition $|\varepsilon|=0$.

An infinite sequence of letters from $\A$ is called \emph{infinite word} and the set of all infinite words over $\A$ is denoted by $\A^\N$.

A finite word $w$ is a \emph{factor} of a (finite or infinite) word $v$ if there exists a finite word $x$ and a (finite or infinite, respectively) word $y$ such that $v=xwy$. The word $w$ is called a \emph{prefix} of $v$ if $x=\varepsilon$, and a \emph{suffix} of $v$, if $y=\varepsilon$.

An infinite word $v$ is \emph{recurrent} if any factor of $v$ occurs infinitely often in $v$.

If $w\in\A^*$ and $k\in\N$, $w^k$ stands for the concatenation of $k$ words $w$, thus $w^k=\underbrace{ww\cdots w}_{k\text{ times } w}$. We also set $w^0=\varepsilon$.
One can introduce negative powers as well. If a word $v\in\A^*$ has the prefix $w^k$, $k\in\N$, then the symbol $w^{-k}v$ denotes the word satisfying $w^kw^{-k}v=v$. Similarly, if a $v\in\A^\N$ has the suffix $w^k$ for a $k\in\N$, then $vw^{-k}$ denotes the word with the property $vw^{-k}w^k=v$.

\subsection{Parikh vectors, abelian complexity, and relative Parikh vectors}

Let $\A=\{0,1,2,\ldots,m-1\}$. For any $\ell\in\A$ and for any $w\in\A^*$, the symbol $|w|_\ell$ denotes the number of occurences of the letter $\ell$ in the word $w$. The \emph{Parikh vector} of $w$ is the $m$-tuple $\Psi(w)=(|w|_0,|w|_1,\ldots,|w|_{m-1})$; note that $|w|_0+|w|_1+\cdots+|w|_{m-1}=|w|$.

For any given infinite word $\u$, we set
$$
\mathcal{P}_\u(n)=\left\{\Psi(w)\,\left|\,\text{$w$ is a factor of $\u$}, |w|=n\right.\right\},
$$
thus $\mathcal{P}_\u(n)$ denotes the set of all Parikh vectors corresponding to factors of $\u$ having the length $n$.
\emph{Abelian complexity} of the word $\u$ is the function $\AC_\u:\N\to\N$ defined as
\begin{equation}\label{AC}
\AC_\u(n)=\#\P_\u(n)\,,
\end{equation}
where $\#$ denotes the cardinality.
Let us introduce two new terms:

\begin{definition}
Let $\u_{[n]}$ denote the prefix of $\u$ of the length $n\in\N_0$.
\begin{itemize}
\item If $w$ is a factor of $\u$ of the length $n$, then the \emph{relative Parikh vector} of $w$ is defined as
$$
{\Psi}^\mathrm{rel}(w)=\Psi(w)-\Psi(\u_{[n]})\,.
$$
\item The \emph{set of relative Parikh vectors} corresponding to the length $n$ is the set
$$
\Pred_\u(n):=\left\{\left.{\Psi}^\mathrm{rel}(w)\;\right|\;\text{$w$ is a factor of $\u$}, |w|=n\right\}\,.
$$
\end{itemize}
\end{definition}

\begin{remark}
The idea of transforming the Parikh vectors $\Psi(w)$ into the relative Parikh vectors ${\Psi}^{\mathrm{rel}}(w)$ slightly resembles a technique of Adamczewski used in~\cite{adamczewski}, where frequencies of letters were employed for a simplification of the study of balance properties of fixed points of primitive substitutions.
\end{remark}

Since prefixes of $\u$ will play an important role in the sequel, the symbol $\u_{[n]}$ will be used in the same meaning throughout the whole paper.

Note that the cardinality of $\Pred_\u(n)$ is equal to the cardinality of $\P_\u(n)$, whence we obtain, with regard to~\eqref{AC},
\begin{equation}\label{ACred}
\AC_\u(n)=\#\Pred_\u(n)\,.
\end{equation}

An infinite word $\u$ is said to be \emph{$c$-balanced}, if for every $\ell\in\A$ and for every pair of factors $v$, $w$ of $\u$ such that $|v|=|w|$, it holds
$\left||v|_\ell-|w|_\ell\right|\leq c$.

\begin{observation}\label{ObsPred}
For all $n\in\N$, the set of relative Parikh vectors has the following properties:
\begin{itemize}
\item[(i)] $\vec{0}\in\Pred_\u(n)$ for all $n$.
\item[(ii)] If $(\psi'_0,\psi'_1,\ldots,\psi'_{m-1})\in\Pred_\u(n)$ and $\u$ is $c$-balanced, then $|\psi'_\ell|\leq c$ for all $\ell\in\A$.
\end{itemize}
\end{observation}

\begin{proof}
\begin{itemize}
\item[(i)]
It is easy to see that $\vec{0}$ is the relative Parikh vector of $\u_{[n]}$, hence $\vec{0}\in\Pred_\u(n)$.
\item[(ii)] Let $w$ be a factor of $\u$ such that $(\psi'_0,\psi'_1,\ldots,\psi'_{m-1})={\Psi}^\mathrm{rel}(w)$, $|w|=n$. Then $|\psi'_\ell|=\left||w|_\ell-|\u_{[n]}|_\ell\right|$. Since $\u$ is $c$-balanced, for any pair $v,w$ of factors of $\u$, it holds $\left||w|_\ell-|v|_\ell\right|\leq c$ for all $\ell\in\A$. Particular choice $v=\u_{[n]}$ gives the statement (ii).
\end{itemize}
\end{proof}

\begin{remark}
Let $\u$ be a $c$-balanced word. With regard to the part (ii) of Observation~\ref{ObsPred}, the main advantage of dealing with relative Parikh vectors instead of with ``standard'' Parikh vectors is twofold:
\begin{itemize}
\item The components of $\Psi(w)$ grow to infinity with growing $|w|$ (because $|w|_0+|w|_1+\cdots+|w|_{m-1}=|w|$), whereas the components of ${\Psi}^\mathrm{rel}(w)$ are bounded.
\item The set of all Parikh vectors $\left\{\left.\Psi(w)\right|\text{$w$ is a factor of $\u$}\right\}$ is infinite, whereas the set of all relative Parikh vectors $\left\{\left.{\Psi}^\mathrm{rel}(w)\right|\text{$w$ is a factor of $\u$}\right\}$ is finite. (This is a very important advantage.)
\end{itemize}
\end{remark}

\subsection{Fixed points of substitutions, normal $F$-representation}

A mapping $\varphi:\A^*\to\A^*$ is called a \emph{morphism} if $\varphi(vw)=\varphi(v)\varphi(w)$ for all $v,w\in\A^*$. A morphism is fully determined if we define $\varphi(\ell)$ for all $\ell\in\A$.
A morphism is called a \emph{substitution} if $\varphi(\ell)\neq\varepsilon$ for all $\ell\in\A$ and at the same time there is an $\ell'\in\A$ such that $|\varphi(\ell')|>1$.
The action of $\varphi$ can be naturally extended to infinite words in the way
$$
\varphi(u_0u_1u_2\cdots)=\varphi(u_0)\varphi(u_1)\varphi(u_2)\cdots\,.
$$
An infinite word $\u$ is called a \emph{fixed point} of the substitution $\varphi$ if $\varphi(\u)=\u$.

For a given substitution $\varphi$, let us set $F_k=|\varphi^k(0)|$ for every $k\in\N_0$. The sequence $(F_k)_{k=0}^{\infty}$ is strictly increasing, thus allows to construct the \emph{normal $F$-representations} (cf.~\cite{Lo}) of positive integers. For any $n\in\N_0$, the normal $F$-representation of $n$, denoted by $\langle n\rangle_F$, takes the form
\begin{equation}\label{Urepr}
\langle n \rangle_F=(d_N, d_{N-1}, \ldots, d_1, d_0)\,,
\end{equation}
where $n=\sum_{i=0}^N d_i F_i$, and moreover, the coefficients $d_i$ are obtained by the greedy algorithm:
\begin{enumerate}
\item Find $N\in\N_0$ such that $n<F_{N+1}$.
\item Put $x_{N}:=n$.
\item For $i=N, N-1,\ldots,1,0$ set $d_i:=\left\lfloor\frac{x_i}{F_i}\right\rfloor$ and $x_{i-1}:=x_i-d_iF_i$.
\end{enumerate}
Note that the number $N$ can be chosen as any integer such that $n<F_{N+1}$, i.e., not necessarily the smallest one satisfying $F_N\leq n<F_{N+1}$. The normal $F$-representation of $n$ can therefore begin with a block of zeros, and for the same reason, the normal $F$-representations $(0,0,...,0,d_N,d_{N-1},...,d_0)$ and $(d_N,d_{N-1},...,d_0)$ are equivalent.


\subsection{Infinite words associated with Parry numbers}

An important class of infinite words is represented by fixed points of substitutions associated with Parry numbers. There exist two types of them, see~\cite{Fa}:

\begin{itemize}
\item Infinite words associated with \emph{simple Parry numbers} are fixed points of substitutions of the type
\begin{equation}\label{simpleParry}
\begin{array}{ccl}
0 & \mapsto & 0^{\alpha_0}1 \\
1 & \mapsto & 0^{\alpha_1}2 \\
 & \vdots & \\
m-2 & \mapsto & 0^{\alpha_{m-2}}(m-1) \\
m-1 & \mapsto & 0^{\alpha_{m-1}}
\end{array}
\end{equation}
\item Infinite words associated with \emph{non-simple Parry numbers} are fixed points of substitutions of the type
\begin{equation}\label{nonsimpleParry}
\begin{array}{ccl}
0 & \mapsto & 0^{\alpha_0}1 \\
1 & \mapsto & 0^{\alpha_1}2 \\
 & \vdots & \\
m & \mapsto & 0^{\alpha_{m}}(m+1) \\
 & \vdots & \\
m+p-2 & \mapsto & 0^{\alpha_{m+p-2}}(m+p-1) \\
m+p-1 & \mapsto & 0^{\alpha_{m+p-1}}m
\end{array}
\end{equation}
\end{itemize}

The exponents $\alpha_j$ in~\eqref{simpleParry} and \eqref{nonsimpleParry} obey certain conditions (see~\cite{Pa,Fa}), in particular:
\begin{equation}\label{alpha}
\begin{array}{cc}
\text{both substitutions \eqref{simpleParry}, \eqref{nonsimpleParry}:} & \alpha_0\geq1 \quad\text{and}\quad \alpha_\ell\leq \alpha_0 \;\text{ for all $\ell\in\A$} \\
\text{substitution \eqref{nonsimpleParry}:} & (\exists \ell\in\{m,m+1,\ldots,m+p-1\})(\alpha_\ell\geq1)
\end{array}
\end{equation}
Hence we obtain an observation:

\begin{observation}\label{>=|A|}
For any $\ell\in\A$, the word $\varphi^{m}(\ell)$ for $\varphi$ given by~\eqref{simpleParry} and the word $\varphi^{m+p}(\ell)$ for $\varphi$ given by~\eqref{nonsimpleParry} begin with $0$.
\end{observation}

For both substitutions~\eqref{simpleParry} and~\eqref{nonsimpleParry}, the corresponding fixed points can be formally written as $\u=\lim_{k\to\infty}\varphi^k(0)$.
We will mostly call them \emph{Parry words}, and where appropriate, the notions \emph{simple Parry word} and \emph{non-simple Parry word} will be used.

\begin{remark}
Denoting the letters by numeric symbols $0,1,2,\ldots$, such as in~\eqref{simpleParry} and~\eqref{nonsimpleParry}, allows to identify every letter $\ell\in\A$ with the corresponding integer $\ell\in\N_0$. Thanks to this conveninent feature we can, for instance, regard the symbol $\ell+j$ as a letter corresponding to the value of the sum.
\end{remark}

For any substitution~\eqref{simpleParry} and~\eqref{nonsimpleParry}, a normal $F$-representation~\eqref{Urepr} of integers can be constructed.
Its coefficients $d_i$ are all less than or equal to $\alpha_0$, which follows from the inequality $F_{N+1}=|(\varphi^N(0))^{\alpha_0}\varphi^N(1)|<|(\varphi^N(0))^{\alpha_0+1}|=(\alpha_0+1)F_N$ holding for all $N\in\N_0$ (recall that $\alpha_1\leq\alpha_0$, see~\eqref{alpha}).

Importantly, if $\u$ is a Parry word, then $\langle n \rangle_F$ essentially describes the structure of $\u_{[n]}$, cf.~\cite{Fa}:

\begin{proposition}\label{struktura pref. u}
Let $\u$ be a fixed point of the substitution $\varphi$ associated to a Parry number. Let $\u_{[n]}$ be a prefix of $\u$ of the length $n$, $n\in\N$. If $\langle n\rangle_F=(d_{N},d_{N-1},\ldots,d_1,d_0)$, then
$\u_{[n]}=\left(\varphi^{N}(0)\right)^{d_{N}}\left(\varphi^{N-1}(0)\right)^{d_{N-1}}\cdots\left(\varphi(0)\right)^{d_1}0^{d_0}$.
\end{proposition}

\section{Abelian co-decomposition}\label{MainIdea}

The method for calculating $\AC(n)$ we are going to establish is based on counting the number of relative Parikh vectors (see~\eqref{ACred}). In this section we present its keynote, further development will follow in subsequent sections.

\begin{observation}\label{B(n)}
Let $\u$ be a recurrent infinite word. Then for any $n\in\N$ there exists a number $B(n)$ such that the prefix $\u_{[B(n)]}$ of $\u$ has these properties:
\begin{itemize}
\item $\u_{[B(n)]}$ contains all factors of $\u$ of the length $n$,
\item $\u_{[B(n)]}$ begins and ends with $\u_{[n]}$.
\end{itemize}
\end{observation}

\begin{proof}
The word $\u$ has finitely many factors of the length $n$ (their number is bounded by $(\#\A)^n$), thus there is an $R'(n)\in\N$ such that all these factors occur in the prefix $\u_{[R'(n)]}$ (cf. also~\cite{AB}). Since $\u$ is recurrent, it is possible to extend $\u_{[R'(n)]}$ to the right to get a longer prefix $\u_{[B(n)]}$ of $\u$ which ends with $\u_{[n]}$.
\end{proof}

From now on let $B(n)$ have the meaning introduced in Observation~\ref{B(n)}. Note that $B(n)$ (and $\u_{[B(n)]}$) depends on the given recurrent word $\u$.

\begin{proposition}\label{posun.}
Let $\u$ be a recurrent word. A word $w$ is a factor of $\u$ of the length $n$ if and only if it can be written as
\begin{equation}\label{wzuy}
w=x^{-1}\u_{[n]}y\,,
\end{equation}
where $x$ is a prefix of $\u_{[B(n)]}\u_{[n]}^{-1}$, $y$ is the prefix of $\u_{[n]}^{-1}\u_{[B(n)]}$, and $|x|=|y|$.
\end{proposition}

\begin{proof}
Let $w$ be a factor of $\u$ of the length $n$.
It follows from the definition of $B(n)$ that $w$ is a factor of $\u_{[B(n)]}$, therefore $\u_{[B(n)]}$ has a prefix $xw$ for a certain word $x$. At the same time $\u_{[B(n)]}$ has the prefix $\u_{[n]}$. Since $\u_{[n]}$ is not longer than $xw$ (because $|xw|\geq|w|=n$), the word $\u_{[n]}$ is a prefix of $xw$, thus there exists a word $y$ such that $xw=\u_{[n]}y$. This equality implies $w=x^{-1}\u_{[n]}y$ and (equivalently) $y=\u_{[n]}^{-1}xw$, hence $y$ is a prefix of $\u_{[n]}^{-1}\u_{[B(n)]}$ of the length $|y|=|xw|-|\u_{[n]}|=|x|+n-n=|x|$.

Conversely, if a word $w$ is given by~\eqref{wzuy}, then obviously $|w|=n$ and $w$ is a factor of $\u_{[B(n)]}$, therefore $w$ is a factor of $\u$ of the length $n$.
\end{proof}

\begin{corollary}\label{Pred.}
Let $\u$ be a recurrent word and $n\in\N$. Then
\begin{multline}\label{Pred(n)orig.}
\Pred_\u(n)=\left\{\Psi(y)-\Psi(x)\;\left|\;\text{$x$ is a prefix of $\u_{[B(n)]}\u_{[n]}^{-1}$,}\right.\right. \\
\left.\text{$y$ is a prefix of $\u_{[n]}^{-1}\u_{[B(n)]}$, $|x|=|y|$}\right\}\,.
\end{multline}
\end{corollary}

\begin{proof}
The set $\Pred_\u(n)$ is defined as $\left\{\left.\Psi(w)-\Psi(\u_{[n]})\;\right|\;\text{$w$ is a factor of $\u$}, |w|=n\right\}$.
According to Proposition~\ref{posun.}, factors of $\u$ of the length $n$ are the words given as
$
w=x^{-1}\u_{[n]}y
$,
where $x$ is a prefix of $\u_{[B(n)]}\u_{[n]}^{-1}$, $y$ is a prefix of $\u_{[n]}^{-1}\u_{[B(n)]}$ and $|x|=|y|$.
Therefore $\Psi(w)-\Psi(\u_{[n]})=\Psi(y)-\Psi(x)$, which together with the definition of $\Pred_\u(n)$ gives
the formula~\eqref{Pred(n)orig.}.
\end{proof}

Corollary~\ref{Pred.} transforms the calculation of $\AC(n)$ into a comparison of prefixes of certain words.
Let us proceed to a trivial observation.

\begin{observation}\label{podobne.}
For all $n\in\N$,
it holds
$$
\Psi\left(\u_{[B(n)]}\u_{[n]}^{-1}\right)=\Psi\left(\u_{[n]}^{-1}\u_{[B(n)]}\right)\,.
$$
\end{observation}

\begin{definition}\label{Def.Codec.}
Let $v,w$ be finite words such that $\Psi(v)=\Psi(w)$, and let
\begin{equation}\label{rozklad.}
\begin{array}{ccccccc}
v&=&z_0&z_1&z_2&\cdots&z_h \\
w&=&\tilde{z}_0&\tilde{z}_1&\tilde{z}_2&\cdots&\tilde{z}_h
\end{array}
\end{equation}
for non-empty factors $z_0,z_1,\ldots,z_h$ and $\tilde{z}_0,\tilde{z}_1,\ldots,\tilde{z}_h$ satisfying $\Psi(\tilde{z}_j)=\Psi(z_j)$ for all $j\in\{0,1,\ldots,h\}$. Then the set of ordered pairs
\begin{equation}\label{codec.}
\Codec\begin{pmatrix}v\\ w\end{pmatrix}=\left\{\begin{pmatrix}z_0\\ \tilde{z}_0\end{pmatrix},\begin{pmatrix}z_1\\ \tilde{z}_1\end{pmatrix},\begin{pmatrix}z_2\\ \tilde{z}_2\end{pmatrix},\cdots,\begin{pmatrix}z_h\\ \tilde{z}_h\end{pmatrix}\right\}
\end{equation}
is called \emph{abelian co-decomposition} of the ordered pair $\begin{pmatrix}v\\ w\end{pmatrix}$.
\end{definition}

\begin{remark}\label{rem.codec.}
\begin{itemize}
\item Abelian co-decomposition~\eqref{codec.} is in general not unique.
\item For any $v,w$ such that $\Psi(v)=\Psi(w)$ there exists at least one abelian co-decomposition, namely $\left\{\begin{pmatrix}v\\ w\end{pmatrix}\right\}$. 
\end{itemize}
\end{remark}

With regard to Observation~\ref{podobne.}, the abelian co-decomposition is applicable to the pair of factors $\u_{[B(n)]}\u_{[n]}^{-1}$, $\u_{[n]}^{-1}\u_{[B(n)]}$:

\begin{definition}\label{def.Z(n).}
For a given recurrent word $\u$ and for any $n\in\N$, we denote
\begin{equation}\label{Z(n).}
\Z_\u(n)=\Codec\begin{pmatrix}\u_{[B(n)]}\u_{[n]}^{-1}\\ \u_{[n]}^{-1}\u_{[B(n)]}\end{pmatrix}
\end{equation}
\end{definition}

Note that since $\Codec\begin{pmatrix}\u_{[B(n)]}\u_{[n]}^{-1}\\ \u_{[n]}^{-1}\u_{[B(n)]}\end{pmatrix}$ is not uniquely given, $\Z_\u(n)$ is not uniquely given, too. However, knowing any $\Z_\u(n)$ allows to calculate the set of relative Parikh vectors corresponding to the number $n$, as we will see in Proposition~\ref{prop.Pred(n).} below, and thus to solve the problem of determining $\AC_\u(n)$, because $\AC_\u(n)$ is nothing but the cardinality of $\Pred_\u(n)$, see formula~\eqref{ACred}.

\begin{proposition}\label{prop.Pred(n).}
Let $\u$ be a recurrent word. For any $n\in\N$, it holds
\begin{equation}\label{Pred(n).}
\Pred_\u(n)=\bigcup_{\begin{pmatrix}z\\ \tilde{z}\end{pmatrix}\in\Z_\u(n)}\left\{\Psi(s)-\Psi(r)\;\left|\;\text{$r$ is a prefix of $z$, $s$ is a prefix of $\tilde{z}$, $|s|=|r|$}\right.\right\}\,.
\end{equation}
\end{proposition}

\begin{proof}
With regard to the definition of $\Z_\u(n)$, we can write
\begin{equation}\label{0p.}
\begin{array}{ccccccc}
\u_{[B(n)]}\u_{[n]}^{-1}&=&z_0&z_1&z_2&\cdots&z_{h} \\
\u_{[n]}^{-1}\u_{[B(n)]}&=&\tilde{z}_0&\tilde{z}_1&\tilde{z}_2&\cdots&\tilde{z}_{h}
\end{array}
\end{equation}
where $\begin{pmatrix}z_j\\ \tilde{z}_j\end{pmatrix}\in\Z_\u(n)$ for all $j=0,1,\ldots,h$.

Now we apply formula~\eqref{Pred(n)orig.}. Let $x$ be a prefix of $\u_{[B(n)]}\u_{[n]}^{-1}$, $y$ be the prefix of $\u_{[n]}^{-1}\u_{[B(n)]}$ and $|x|=|y|$. Then, with regard to~\eqref{0p.},
$$
\begin{array}{ccccccc}
x&=&z_0&z_1&\cdots&z_{h'}&r \\
y&=&\tilde{z}_0&\tilde{z}_1&\cdots&\tilde{z}_{h'}&s
\end{array}
$$
for a certain $h'<h$, an $r$ being a prefix of $z_{h'+1}$ and an $s$ being a prefix of $\tilde{z}_{h'+1}$, $|r|=|s|$. Since $\Psi(z_j)=\Psi(\tilde{z}_j)$ for all $j=0,1,\ldots,h$, it holds $\Psi(y)-\Psi(x)=\Psi(s)-\Psi(r)$, thus~\eqref{Pred(n).} is proved.
\end{proof}

Obviously, practical usefulness of the formula~\eqref{Pred(n).} depends on whether one can obtain $\Z_\u(n)$ efficiently. In this paper we will show a solution for infinite words associated with Parry numbers; it is likely that suitable ways can be found also for other families of recurrent words.

\section{On the recurrence of Parry words}\label{recParry}

Any Parry word 
$\u$ 
is recurrent, which can be easily checked directly (and also it follows e.g. from~\cite{Du,Ca}), therefore the method of abelian co-decomposition can be applied.
The aim of this section is to determine, for any given $n\in\N$, the prefix $\u_{[B(n)]}$, introduced in Observation~\ref{B(n)}. Recall that the knowledge of $\u_{[B(n)]}$ is needed to calculate $\Z_\u(n)$, see formula~\eqref{Z(n).}.

Let us start with a trivial observation.

\begin{observation}\label{0x0Obs}
Let $\u$ be a Parry word, i.e., the fixed point of a substitution given by~\eqref{simpleParry} or~\eqref{nonsimpleParry}. If $w$ is a factor of $\u$ of the length $n\leq F_{k}=|\varphi^{k}(0)|$, then $w$ is a factor of $\varphi^{k}(0t0)$ for a certain word $t$ not containing the letter $0$ ($t$ may be the empty word).
\end{observation}

In other words, to find all factors of $\u$ of the length $n\leq F_{k}$, it suffices to explore the factors $\varphi^{k}(0t0)$ for $t$ not containing $0$. Let us denote
$$
\mathcal{T}=\{0t0\,|\,\text{$0t0$ is a factor of $\u$},\,|t|_0=0\},
$$
and find a prefix of $\u$ containing all factors from $\mathcal{T}$.
We will deal with the simple Parry words at first.

\begin{observation}\label{0x0LemmaSim}
Let $\u$ be the fixed point of~\eqref{simpleParry} and $0t0\in\mathcal{T}$.
\begin{itemize}
\item[(i)] If $t\neq\varepsilon$ and $t$ ends with a letter $\ell\in\{1,2,\ldots,m-1\}$, then $0t0$ is a suffix of $\varphi^\ell(0)0$.
\item[(ii)] If $t=\varepsilon$, then $0t0=00$ is a suffix of $\varphi^m(0)0$. If moreover $\alpha_0\geq2$, then $00$ is also a prefix of $\u$.
\end{itemize}
\end{observation}

\begin{proof}
(i) Let $t\neq\varepsilon$ and $\ell$ be its last letter. It follows from~\eqref{simpleParry} that $0t0$ is necessarily a suffix of a factor $\varphi^{\ell+im}(v0)0$ of $\u$, where $v$ is a certain factor and $i\in\N_0$. Since $\varphi^{\ell+im}(v0)0=\varphi^\ell(v'0)0$ for a certain $v'$, we can assume without loss of generality that $i=0$, thus $0t0$ is a suffix of $\varphi^{\ell}(v0)0$.
Finally, since $|\varphi^\ell(0)|_0\geq1$ (recall that $\alpha_0\geq1$) and we assume $|t|_0=0$, we have $|\varphi^\ell(0)0|_0\geq|0t0|_0$, hence $0t0$ is not longer than $\varphi^{\ell}(0)0$, in other words, $0t0$ it is a suffix of $\varphi^{\ell}(0)0$.

(ii) The statement for $t=\varepsilon$ is obvious from~\eqref{simpleParry}.
\end{proof}

\begin{remark}\label{0x0Rem}
Observation~\ref{0x0LemmaSim} implies that any $0t0\in\mathcal{T}$ is a factor of $\varphi^\ell(0)0$ for a certain $\ell$, where $\ell\in\{1,\ldots,m-1\}$ if $\alpha_0\geq2$ and $\ell\in\{1,\ldots,m-1,m\}$ if $\alpha_0=1$.
\end{remark}

In the proposition below we find an $R_{\mathrm{s}}$ such that $\varphi^{R_{\mathrm{s}}}(0)0$ is a \emph{prefix of $\u$} containing \emph{all} factors from $\mathcal{T}$.

\begin{proposition}\label{0x0sim}
Let $\u$ be the fixed point of~\eqref{simpleParry}. Let us set
\begin{equation}\label{Rsim}
R_{\mathrm{s}}=\left\{\begin{array}{ll}
m-1 & \text{if } \alpha_0\geq2, \\
m+\ell'-1, \text{where $\ell'=\min\{\ell\geq1\,|\,\alpha_{\ell}\geq1\}$} & \text{if } \alpha_0=1.
\end{array}\right.
\end{equation}
Then $\varphi^{R_{\mathrm{s}}}(0)0$ is a prefix of $\u$ such that every $0t0\in\mathcal{T}$ is a factor of $\varphi^{R_{\mathrm{s}}}(0)0$.
\end{proposition}

\begin{proof}
Let $\alpha_0\geq2$. The word $\u$ has the prefix $00$, and thus also $\varphi^j(00)=\varphi^j(0)0\cdots$ for all $j\in\N$. Therefore, $\varphi^{\ell}(0)0$ is a prefix of $\u$ for all $\ell=1,2,\ldots,m-1=R_{\mathrm{s}}$. The statement then follows from Remark~\ref{0x0Rem}.

Let $\alpha_0=1$. The word $\u$ has the prefix $01$, thus also $\varphi^j(01)$ for all $j\in\N$. Note that $j\geq \ell'\Rightarrow \varphi^j(1)=0\cdots$, hence for all $j\geq \ell'$, the word $\varphi^j(0)0$ is a prefix of $\u$, in particular for $j=\ell',\ell'+1,\ldots,m$.

The prefix $\varphi^{m}(0)0$ of $\u$ ends with $00$. Therefore
$$
\varphi^j(\varphi^m(0)0)=\varphi^j(\varphi^m(0))\varphi^j(0)=\varphi^j(\cdots 0)0\cdots=\underbrace{\cdots\varphi^j(0)0}_{\varphi^{m+j}(0)0}\cdots\,,
$$
hence for all $j\in\N$, the word $\varphi^{m+j}(0)0$ is a prefix of $\u$ and $\varphi^j(0)0$ is its factor. From now on let us focus on $j=1,2,\ldots,\ell'-1$. All the corresponding prefixes $\varphi^{m+j}(0)0$ of $\u$ are factors of the longest prefix $\varphi^{m+\ell'-1}(0)0$, therefore $\varphi^{m+\ell'-1}(0)0$ is a prefix of $\u$ containing $\varphi^j(0)0$ as a factor for every $j=1,2,\ldots,\ell'-1$.

At the same time, $\varphi^{m+\ell'-1}(0)0$ is longer than $\varphi^j(0)0$ for all $j=\ell',\ldots,m$, and thus contains also $\varphi^j(0)0$ for all $j=\ell',\ldots,m$ (see above).
To sum up, $\varphi^{m+\ell'-1}(0)0$ is a prefix of $\u$, containing $\varphi^j(0)0$ as its factor for all $j=1,\ldots,m$. The statement then follows from Remark~\ref{0x0Rem}.
\end{proof}

Now we repeat the calculation for non-simple Parry words.

\begin{observation}\label{0x0LemmaNonSim}
Let $\u$ be the fixed point of~\eqref{nonsimpleParry} and $h'=\min\{\ell\geq m\,|\,\alpha_{\ell}\geq1\}$. For any $0t0\in\mathcal{T}$, it holds:
\begin{itemize}
\item[(i)] If $t=\varepsilon$, then $0t0=00$ is a prefix of $\u$.
\item[(ii)] If $t\neq\varepsilon$ and its last letter $\ell$ satisfies $\ell<m$ or $\ell>h'$, then $0t0$ is a suffix of $\varphi^\ell(0)0$ and $\varphi^\ell(0)0$ is a factor of $\u$.
\item[(iii)] If $t\neq\varepsilon$ and its last letter $\ell$ satisfies $m\leq \ell\leq h'$, then $0t0$ is a suffix of $\varphi^\ell(0)0$ and $\varphi^\ell(0)0$ is a factor of $\u$, or $0t0$ is a suffix of $\varphi^{\ell+p}(0)0$ and $\varphi^{\ell+p}(0)0$ is a factor of $\u$.
\end{itemize}
\end{observation}

\begin{proof}
At first note that the definition of $h'$ is correct with regard to~\eqref{alpha}.

(i) If $00\in\mathcal{T}$, there exists an $\ell\in\A$ such that $\alpha_\ell\geq2$. Since $\alpha_\ell\leq \alpha_0$ for all $\ell\in\A$ due to~\eqref{alpha}, we have $\alpha_0\geq2$, hence $00$ is a prefix of $\u$.

(ii), (iii)
Considering~\eqref{nonsimpleParry}, $0t0$ is necessarily a suffix of a factor $\varphi^j(v0)0$ of $\u$, where
$$
j=\left\{\begin{array}{cl}
\ell & \text{for $\ell<m$}, \\
\ell+ip \text{ with a certain $i\in\N_0$} & \text{for $\ell\geq m$}.
\end{array}\right.
$$
Since $|\varphi^j(0)|_0\geq1$ (recall that $\alpha_0\geq1$) and we assume $|t|_0=0$, we have $|\varphi^j(0)0|_0\geq|0t0|_0$, hence $0t0$ is not longer than $\varphi^{j}(0)0$, in other words, $0t0$ is a suffix of $\varphi^{j}(0)0$ and one can consider $v=\varepsilon$.
In the rest of the proof we explain why we can assume $i=0$ for $\ell>h'$ and $i\leq1$ for $m\leq \ell\leq h'$.

Let $\ell>h'$ and $0t0$ be a suffix of $\varphi^{\ell+ip}(0)0$ for an $i>0$. Our goal is to show
that $0t0$ is a suffix of $\varphi^{\ell}(0)0$ and at the same time, $\varphi^{\ell}(0)0$ is a factor of $\u$.
Applying the substitution~\eqref{nonsimpleParry}, we find that $0t0$ is a suffix of
$$
\varphi^{\ell+ip}(0)0=\varphi^{\ell-h'}(\varphi^{ip}(\varphi^{h'}(0)))0=\varphi^{\ell-h'}(\varphi^{ip}(\cdots h'))0=\varphi^{\ell-h'}(\cdots h')0=\cdots\varphi^{\ell-h'}(h')0\,.
$$
It holds $|0t0|_0=2$ (we assume $|t|_0=0$) and $|\varphi^{\ell-h'}(h')0|_0\geq2$ (because the definition of $h'$ implies $|\varphi^{\ell-h'}(h')|_0\geq1$), hence $|0t0|_0\leq|\varphi^{\ell-h'}(h')0|_0$. Consequently, $0t0$ is a suffix of $\varphi^{\ell-h'}(h')0$.
Since $\varphi^{h'}(0)=\cdots h'$, the word
$0t0$
is a suffix of $\varphi^{\ell-h'}(\varphi^{h'}(0))0=\varphi^\ell(0)0$.
Finally, $\varphi^\ell(0)0$ is a prefix of $\u$ due to $\ell\geq h'$.

Let $m\leq \ell\leq h'$ and $0t0$ be a suffix of $\varphi^{\ell+ip}(0)0$ for an $i>1$. We will show that $i$ can be set to $1$.
It holds
$$
\varphi^{\ell+ip}(0)0=\varphi^{p}(\varphi^{(i-1)p}(\varphi^{\ell}(0)))0=\varphi^{p}(\varphi^{(i-1)p}(\cdots \ell))0=\varphi^{p}(\cdots \ell)0=\cdots\varphi^{p}(\ell)0\,.
$$
The assumption $\ell\geq m$ together with conditions~\eqref{alpha} imply $|\varphi^{p}(\ell)0|_0\geq2$, and since at the same time $|0t0|_0=2$, we have $|0t0|_0\leq|\varphi^{p}(\ell)0|_0$, hence $0t0$ is a suffix of $\varphi^{p}(\ell)0$.
The letter $\ell$ is a suffix of $\varphi^{\ell}(0)$, thus $0t0$ is a suffix of $\varphi^{\ell+p}(0)0$. Finally, the word $\varphi^{\ell+p}(0)0$ is a prefix, and thus a factor, of $\u$ due to $\ell+p\geq m+p$ and, e.g., Observation~\ref{>=|A|}.
\end{proof}

\begin{remark}\label{0x0RemNs}
Observation~\ref{0x0LemmaNonSim} implies that for any $0t0\in\mathcal{T}$ there exists a $j\in\{1,2,\ldots,m+p-1\}\cup\{m+p,m+1+p,\ldots,h'+p\}=\{1,2,\ldots,h'+p\}$ such that $0t0$ is a factor of $\varphi^j(0)0$ and $\varphi^j(0)0$ is a factor of $\u$.
\end{remark}

\begin{proposition}\label{0x0nonsim}
Let $\u$ be the fixed point of~\eqref{nonsimpleParry} and $h'=\min\{\ell\geq m\,|\,\alpha_{\ell}\geq1\}$.
Let us set
\begin{equation}\label{Rnonsim}
R_{\mathrm{ns}}=\left\{\begin{array}{ll}
h'+p & \text{if } \alpha_0\geq2, \\
h'+m+p-1 & \text{if } \alpha_0=1.
\end{array}\right.
\end{equation}
Then $\varphi^{R_{\mathrm{ns}}}(0)0$ is a prefix of $\u$ such that every $0t0\in\mathcal{T}$ is a factor of $\varphi^{R_{\mathrm{ns}}}(0)0$.
\end{proposition}

\begin{proof}
With regard to Remark~\ref{0x0RemNs}, it suffices to 
prove that $\varphi^{R_{\mathrm{ns}}}(0)0$ is a prefix of $\u$ and to
verify the implication ($j\leq h'+p$ and $\varphi^j(0)0$ is a factor of $\u$) $\Rightarrow$ $\varphi^j(0)0$ is a factor of $\varphi^{R_{\mathrm{ns}}}(0)0$.

Let $\alpha_0\geq2$. The word $\u$ has the prefix $00$, and thus also $\varphi^j(00)=\varphi^j(0)0\cdots$ for all $j\in\N$. Consequently, $\varphi^{R_{\mathrm{ns}}}(0)0$ for $R_{\mathrm{ns}}=h'+p$ is a prefix of $\u$ which trivially contains all the required shorter prefixes $\varphi^j(0)0$ for $j\leq h'+p$.

Let $\alpha_0=1$. The word $\u$ has the prefix $01$, and thus also $\varphi^j(01)$ for all $j\in\N$. Note that $j\geq h'\Rightarrow \varphi^j(1)=0\cdots$, hence $\varphi^j(0)0$ is a prefix of $\u$ for all $j\geq h'$. Since $R_{\mathrm{ns}}=h'+m+p-1\geq h'$, the word $\varphi^{R_{\mathrm{ns}}}(0)0$ is a prefix of $\u$, and $\varphi^j(0)0$ is a prefix of $\varphi^{R_{\mathrm{ns}}}(0)0$ for all $j=h',h'+1,\ldots,h'+p$.

Let us proceed to $j<h'$. Generally, if $\varphi^{j}(0)0$ is a factor of $\u$, then there is an $\ell\in\A$ such that $\varphi^j(0)0$ is a prefix of $\varphi^j(0\ell)$. We see from~\eqref{nonsimpleParry} that $0\ell$ is a suffix of $\varphi(\hat{\ell})$ (where $\hat{\ell}=\ell-1$ for $\ell\neq m$ and $\hat{\ell}\in\{m-1,m+p-1\}$ for $\ell=m$), hence $\varphi^j(0)0$ is a factor of $\varphi^{j+1}(\hat{\ell})$. Since $\hat{\ell}$ is a suffix of $\varphi^{\hat{\ell}}(0)$, we conclude that $\varphi^j(0)0$ is a factor of $\varphi^{j+1+\hat{\ell}}(0)$. Finally, $\hat{\ell}\in\A$ implies $\hat{\ell}\leq m+p-1$, hence $j+1+\hat{\ell}\leq h'-1+1+m+p-1=R_{\mathrm{ns}}$, therefore $0t0$ is a factor of $\varphi^{R_{\mathrm{ns}}}(0)0$.
\end{proof}

In the proposition below we find the formula for $\u_{[B(n)]}$, valid for any Parry word.

\begin{proposition}\label{vsechny}
Let $\u$ be a Parry word and
\begin{equation}\label{R}
R=\left\{\begin{array}{cl}
R_{\mathrm{s}} \; \text{ from~\eqref{Rsim}} & \text{if $\u$ is a fixed point of~\eqref{simpleParry}}, \\
R_{\mathrm{ns}} \; \text{ from~\eqref{Rnonsim}} & \text{if $\u$ is a fixed point of~\eqref{nonsimpleParry}}.
\end{array}\right.
\end{equation}
For any number $n$ we can set
\begin{equation}\label{B(n)Parry}
\u_{[B(n)]}:=\varphi^{k+R}(0)\u_{[n]}\,,
\end{equation}
where
$k$ is any integer such that $n\leq F_{k}$.
\end{proposition}

\begin{proof}
We shall show that $\varphi^{k+R}(0)\u_{[n]}$ is a prefix of $\u$ containing all factors of $\u$ of the length $n$.

Due to Observation~\ref{0x0Obs}, any factor of $\u$ of the length $n\leq F_{k}$ is a factor of $\varphi^{k}(0t0)$ for a certain $0t0\in\mathcal{T}$.
According to Propositions~\ref{0x0sim} or \ref{0x0nonsim}, every $0t0\in\mathcal{T}$ is a factor of $\varphi^{R}(0)0$, and $\varphi^{R}(0)0$ is a prefix of $\u$.
Putting these facts together, we see that the the word $\varphi^k(\varphi^R(0)0)=\varphi^{k+R}(0)\varphi^k(0)$ is a prefix of $\u$ containing all factors of $\u$ of the length $n$.
Since $\varphi^{k+R}(0)\varphi^k(0)$ begins (and ends) with $\varphi^{k}(0)$ and it holds $|\varphi^{k}(0)|\geq n$, we can write
$$
\varphi^{k+R}(0)\varphi^k(0)=\varphi^{k}(0)\cdots\cdots\varphi^k(0)=\varphi^k(0)\cdots\cdots\u_{[n]}\left(\u_{[n]}^{-1}\varphi^{k}(0)\right).
$$
Now we note that
any factor of the length $n$ of this word which 
at least partially overlaps with the suffix $\u_{[n]}^{-1}\varphi^{k}(0)$ is fully contained in the suffix $\varphi^k(0)$ (it follows from $|w|\leq n$), and, therefore, it is fully contained also in the prefix $\varphi^k(0)$.
Consequently, the suffix $\u_{[n]}^{-1}\varphi^{k}(0)$ is redundant and can be left out.
To sum up,
$\varphi^{k+R}(0)\u_{[n]}$ is a prefix of $\u$ ending with $\u_{[n]}$ and containing all factors of $\u$ of the length $n$.
\end{proof}


\section{Properties of the abelian co-decomposition applied to Parry words}

Let $\u$ denote a Parry word, i.e., $\u=\lim_{k\to\infty}\varphi^k(0)$, where $\varphi$ is given by~\eqref{simpleParry} or by~\eqref{nonsimpleParry}, and let $\langle n\rangle_F$ be the corresponding normal $F$-representation of $n\in\N_0$. 
In this section we derive several relations useful for calculating the sets $\Z_\u(n)$ for $n\in\N$ that result from
the relation between the structure of $\u_{[n]}$ and $\langle n\rangle_F$, described in Proposition~\ref{struktura pref. u}.

\begin{proposition}\label{NMn}
Let the normal $F$-representation of an $n\in\N$ take the form
$$
\langle n\rangle_F=(d_{k+K-1},d_{k+K-2},\ldots,d_{k},d_{k-1},\ldots,d_1,d_0)\,,
$$
and let $Q$ and $q$ be the numbers given by the normal $F$-representations
$$
\langle Q\rangle_F=(d_{k+K-1},d_{k+K-2},\ldots,d_{k}),\qquad \langle q\rangle_F=(d_{k-1},d_{k-2},\ldots,d_1,d_0)\,.
$$
If a $\Z_\u(Q)$ has the property
\begin{equation}\label{podm.}
\text{$\varphi^k(\tilde{z})$ has the prefix $\u_{[q]}$ for all $\begin{pmatrix}z\\ \tilde{z}\end{pmatrix}\in\Z_\u(Q)$}\,,
\end{equation}
then
\begin{equation}\label{Z NMn}
\Z_\u(n)=\bigcup_{\begin{pmatrix}z\\ \tilde{z}\end{pmatrix}\in\Z_\u(Q)}\Codec\begin{pmatrix}\varphi^k(z)\\ \u_{[q]}^{-1}\varphi^k(\tilde{z})\u_{[q]}\end{pmatrix}\,.
\end{equation}
\end{proposition}

\begin{proof}
It follows from the assumptions that $n\leq F_{k+K}$, $Q\leq F_{K}$, and with regard to Proposition~\ref{struktura pref. u} also $\u_{[n]}=\varphi^{k}(\u_{[Q]})\u_{[q]}$.

The set $\Z_\u(Q)$ is defined as an abelian co-decomposition of $\begin{pmatrix}\u_{[B(Q)]}\u_{[Q]}^{-1}\\ \u_{[Q]}^{-1}\u_{[B(Q)]}\end{pmatrix}$ (see~\eqref{Z(n).}). Since $Q\leq F_{K}$, we have from Proposition~\ref{vsechny} the formula $\u_{[B(Q)]}=\varphi^{R+K}(0)\u_{[Q]}$. Therefore we can decompose
\begin{equation*}
\begin{array}{ccccccc}
\varphi^{K+R}(0)&=&z_1&z_2&z_3&\cdots&z_{h} \\
\u_{[Q]}^{-1}\varphi^{K+R}(0)\u_{[Q]}&=&\tilde{z}_1&\tilde{z}_2&\tilde{z}_3&\cdots&\tilde{z}_{h}
\end{array}
\end{equation*}
where $\begin{pmatrix}z_j\\ \tilde{z}_j\end{pmatrix}\in\Z_\u(Q)$ for all $j=1,\ldots,h$.
Hence
$$
\varphi^{k+K+R}(0)=\varphi^k(z_1z_2z_3\cdots z_{h})\,,
$$
and, with regard to the abovementionned equality $\u_{[n]}=\varphi^{k}(\u_{[Q]})\u_{[q]}$,
\begin{multline*}
\u_{[n]}^{-1}\varphi^{k+K+R}(0)\u_{[n]}=\u_{[q]}^{-1}\varphi^{k}(\u_{[Q]})^{-1}\varphi^{k+K+R}(0)\varphi^{k}(\u_{[Q]})\u_{[q]}=\\
=\u_{[q]}^{-1}\varphi^{k}(\u_{[Q]}^{-1}\varphi^{K+R}(0)\u_{[Q]})\u_{[q]}=\u_{[q]}^{-1}\varphi^k(\tilde{z}_1\tilde{z}_2\tilde{z}_3\cdots\tilde{z}_{h})\u_{[q]}\,.
\end{multline*}
Recall that $\varphi^k(\tilde{z_j})$ have the prefix $\u_{[q]}$ for all $j$ due to the assumptions, thus we can write
\begin{equation*}
\begin{array}{cccccc}
\varphi^{k+K+R}(0)&=&\varphi^k(z_1)&\varphi^k(z_2)&\cdots&\varphi^k(z_{h}) \\
\u_{[n]}^{-1}\varphi^{k+K+R}(0)\u_{[n]}&=&\u_{[q]}^{-1}\varphi^k(\tilde{z}_1)\u_{[q]}&\u_{[q]}^{-1}\varphi^k(\tilde{z}_2)\u_{[q]}&\cdots&\u_{[q]}^{-1}\varphi^k(\tilde{z}_{h})\u_{[q]}
\end{array}
\end{equation*}
Since $\Psi(\varphi^k(z_j))=\Psi(\u_{[q]}^{-1}\varphi^k(\tilde{z}_j)\u_{[q]})$ for all $j$, the co-decomposition can be applied to those pairs as well, thus
$$
\Codec\begin{pmatrix}\varphi^{k+K+R}(0)\\ \u_{[n]}^{-1}\varphi^{k+K+R}(0)\u_{[n]}\end{pmatrix}=\Codec\begin{pmatrix}z_1\\ \u_{[q]}^{-1}\varphi^k(\tilde{z}_1)\u_{[q]}\end{pmatrix}\cup\cdots\cup\Codec\begin{pmatrix}z_{h}\\ \u_{[q]}^{-1}\varphi^k(\tilde{z}_{h})\u_{[q]}\end{pmatrix},
$$
hence we obtain~\eqref{Z NMn}.
\end{proof}

\begin{remark}
The assumption ``$\varphi^k(\tilde{z})$ has the prefix $\u_{[q]}$ for all $\begin{pmatrix}z\\ \tilde{z}\end{pmatrix}\in\Z_\u(Q)$'' in Proposition~\ref{NMn} is technical and can be always satisfied by a suitable $\Z_\u(Q)$.
\end{remark}

\begin{corollary}
Let the number $n$ satisfy $\langle n\rangle_F=(\langle Q\rangle_F,\underbrace{0,0,\ldots,0}_{\text{$k$ times $0$}})$.
Then
\begin{equation}\label{Z M0}
\Z_\u(n)=\bigcup_{\begin{pmatrix}z\\ \tilde{z}\end{pmatrix}\in\Z_\u(Q)}\Codec\begin{pmatrix}\varphi^k(z)\\ \varphi^k(\tilde{z})\end{pmatrix}\,.
\end{equation}
\end{corollary}

\begin{proof}
The statement is obtained from Proposition~\ref{NMn} for $q=0$.
\end{proof}

\begin{proposition}\label{N1N2}
Let
$$
\langle n_1\rangle_F=(\langle Q_1\rangle_F,\underbrace{d_{k-1},d_{k-2},\ldots,d_1,d_0}_{\langle q\rangle_F}) \quad\text{and}\quad \langle n_2\rangle_F=(\langle Q_2\rangle_F,\underbrace{d_{k-1},d_{k-2},\ldots,d_1,d_0}_{\langle q\rangle_F})\,,
$$
where $\Z_\u(Q_1),\Z_\u(Q_2)$ satisfy
\begin{itemize}
\item[(i)] $\Z_\u(Q_1)=\Z_\u(Q_2)$,
\item[(ii)] $\varphi^k(\tilde{z})$ has the prefix $\u_{[q]}$ for all $\begin{pmatrix}z\\ \tilde{z}\end{pmatrix}\in\Z_\u(Q_1)$.
\end{itemize}
Then $\Pred(n_1)=\Pred(n_2)$.
\end{proposition}

\begin{proof}
Equation~\eqref{Z NMn} in Proposition~\ref{NMn} implies $\Z_\u(n_1)=\Z_\u(n_2)$, and equation~\eqref{Pred(n).} then gives $\Pred(n_1)=\Pred(n_2)$.
\end{proof}

\begin{remark}\label{infinitely}
Proposition~\ref{N1N2} can be used for proving that a certain value of $\AC_\u$ is attained infinitely many times. Let $\langle n_i\rangle_F=(\langle Q_i \rangle_F, \langle q \rangle_F)$, where $\langle q \rangle_F$ is the normal $F$-representation of a fixed number $q$ and $\{Q_i\}_{i=0}^{\infty}$ is a sequence with the property $\Z_\u(Q_i)=\Z_\u(Q_{1})$ for all $i\in\N$, and moreover satisfying (ii) from Proposition~\ref{N1N2}. Then $\Pred(n_i)$ is constant for all $i\in\N$, i.e., $\AC$ attains the value $\AC(n_{1})$ infinitely many times.
It will be demonstrated on a concrete example in Section~\ref{Applications}.
\end{remark}


\section{Algorithm for calculating $\AC_\u(n)$}\label{Algorithm}

Let us summarize the results in the form of an explicit algorithm.

Let $\u$ be a Parry word.
\begin{enumerate}
\item
Find $R$ from~\eqref{R}.
\item Find $\langle n\rangle_F=(d_{k-1},d_{k-2},\ldots,d_1,d_0)$.
\item Let $\langle N_1\rangle_F=(d_{k-1})$. Find $\Z_\u(N_1)$ from~\eqref{Z(n).} using~\eqref{B(n)Parry}:
$$
\Z_\u(N_1)=
\Codec\begin{pmatrix}\varphi^{1+R}(0)\\ 0^{-d_{k-1}}\varphi^{1+R}(0)0^{d_{k-1}}\end{pmatrix}\,.
$$
\emph{N.B.:} $\Z_\u(N_1)$ is not unique; with the view of a correct functionality in the next step, choose the decomposition
$$
\begin{array}{ccccccc}
\varphi^{1+R}(0)&=&z_0&z_1&z_2&\cdots&z_{h} \\
0^{-d_{k-1}}\varphi^{1+R}(0)0^{d_{k-1}}&=&\tilde{z}_0&\tilde{z}_1&\tilde{z}_2&\cdots&\tilde{z}_{h}
\end{array}
$$
so that for all $j\in\{1,\ldots,h\}$, $\varphi(\tilde{z}_h)$ has the prefix $0^{\alpha_0}$. For the optimality, take $h\in\N_0$ maximal possible.
\item For all $i=2,\ldots,k$ do \\
Let $\langle N_i\rangle_F=(d_{k-1},\ldots,d_{k-i})$. Find $\Z_\u(N_i)$ by substituting $Q:=N_{i-1}$, $\langle q\rangle_F:=(d_{k-i})$ into~\eqref{Z NMn}:
$$
\Z_\u(N_i)=
\bigcup_{\begin{pmatrix}z\\ \tilde{z}\end{pmatrix}\in\Z_\u(N_{i-1})}\Codec\begin{pmatrix}\varphi(z)\\ 0^{-d_{k-i}}\varphi(\tilde{z})0^{d_{k-i}}\end{pmatrix}\,.
$$
\emph{N.B.:} $\Z_\u(N_i)$ is not unique; with the view of a correct functionality in the next step, choose all the decompositions
$$
\begin{array}{cccccc}
\varphi(z)&=&z'_0&z'_1&\cdots&z'_{h_i} \\
0^{-d_{k-i}}\varphi(\tilde{z})0^{d_{k-i}}&=&\tilde{z}'_0&\tilde{z}'_1&\cdots&\tilde{z}'_{h_i}
\end{array}
$$
so that for all $j\in\{1,\ldots,h_i\}$, $\varphi(\tilde{z}'_j)$ has the prefix $0^{\alpha_0}$. For the optimality, take $h_i\in\N_0$ maximal possible.
\item Since $N_k=n$, find $\Pred_\u(n)$ from $\Z_\u(N_k)$ using the formula~\eqref{Pred(n).}.
\item $\AC_\u(n)=\#\Pred_\u(n)$.
\end{enumerate}

\begin{remark}
The requirements in Step 3 and 4 ensure that the condition~\eqref{podm.} is satisfied for all $i$.
\end{remark}


\section{Example: Tribonacci word}\label{Applications}

Let $\mathbf{t}$ denote the Tribonacci word, which is the fixed point of the substitution
\begin{equation}\label{Tribonacci}
\begin{array}{ccl}
0 & \mapsto & 01 \\
1 & \mapsto & 02 \\
2 & \mapsto & 0
\end{array}
\end{equation}
Note that~\eqref{Tribonacci} is a substitution of the type~\eqref{simpleParry}, thus $\mathbf{t}$ is a simple Parry word.

The abelian complexity of the Tribonacci word has been examined by Richomme, Saari and Zamboni in~\cite{RSZ}, where it has been shown that $\AC_\mathbf{t}(n)\in\{3,4,5,6,7\}$ for all $n\in\N$ and that each of these five values is assumed. It has been also proved that the extreme values $3$ and $7$ are attained infinitely often, but the question whether $\AC_\mathbf{t}(n)$ attains also other values ($4,5,6$) infinitely many times remained open. Here we give the answer, applying the method
of abelian co-decomposition developed above.

\subsection{$\AC$ attains the value $4$ infinitely many times}

\begin{proposition}\label{Trib.4}
If $\langle n_i\rangle_F=(\underbrace{1,0,1,0,\ldots,1,0}_{\text{$i$ times $(1,0)$}},1)=((1,0)^i,1)$, then
$$
\begin{array}{ccl}
\Z(n_0) & = & \left\{\begin{pmatrix}01\\10\end{pmatrix},\begin{pmatrix}02\\20\end{pmatrix},\begin{pmatrix}0\\0\end{pmatrix}\right\} \\
\Z(n_1) & = & \left\{\begin{pmatrix}01\\10\end{pmatrix},\begin{pmatrix}0\\0\end{pmatrix},\begin{pmatrix}201\\102\end{pmatrix},\begin{pmatrix}0201\\1020\end{pmatrix},\begin{pmatrix}02\\20\end{pmatrix}\right\} \\
\Z(n_i) & = & \left\{\begin{pmatrix}01\\10\end{pmatrix},\begin{pmatrix}0\\0\end{pmatrix},\begin{pmatrix}201\\102\end{pmatrix},\begin{pmatrix}0201\\1020\end{pmatrix},\begin{pmatrix}02\\20\end{pmatrix},\begin{pmatrix}1\\1\end{pmatrix},\begin{pmatrix}2\\2\end{pmatrix}\right\} \quad\text{for all $i\geq2$}.
\end{array}
$$
\end{proposition}

\begin{proof}
At first, we compute $\Z(n_0)$ from the definition~\eqref{Z(n).}, wherefore we need to determine $\mathbf{t}_{[n_0]}$ and $\mathbf{t}_{[B(n_0)]}$. We have $\langle n_0\rangle_F=(1)$, hence $n_0=1$ and $\mathbf{t}_{[n_0]}=0$. Since $n_0\leq1=F_0$ and equation~\eqref{Rsim} implies $R=m$, we have $\mathbf{t}_{[B(n_0)]}=\varphi^{0+3}(0)0$ from~\eqref{B(n)Parry}. Therefore \eqref{Z(n).} takes the form
$$
\Z(n_0)=\Codec\begin{pmatrix}\varphi^3(0)\\ 0^{-1}\varphi^3(0)0\end{pmatrix}\,.
$$
We calculate
$$
\begin{array}{cccccccc}
\varphi^3(0)&=&0102010 &=& \overbrace{01}^{z_1}&\overbrace{02}^{z_2}&\overbrace{01}^{z_3}&\overbrace{0}^{z_4} \\
0^{-1}\varphi^3(0)0&=&1020100 &=& \underbrace{10}_{\tilde{z}_1}&\underbrace{20}_{\tilde{z}_2}&\underbrace{10}_{\tilde{z}_3}&\underbrace{0}_{\tilde{z}_4}
\end{array}
$$
hence we obtain the decomposition
$$
\Z(n_0)=\left\{\begin{pmatrix}01\\10\end{pmatrix},\begin{pmatrix}02\\20\end{pmatrix},\begin{pmatrix}0\\0\end{pmatrix}\right\}\,.
$$

Let us proceed to $i\geq1$. Since it holds $\langle n_i\rangle_F=(\langle n_{i-1}\rangle_F,0,1)$ for all $i\in\N$, the decomposition $\Z(n_i)$ will be derived from $\Z(n_{i-1})$ using formula~\eqref{Z NMn} with $\langle q\rangle_F=(0,1)$ (whence $\mathbf{t}_{[q]}=0$) and $Q=n_{i-1}$:
\begin{equation}\label{Uk.4}
\Z(n_i)=\bigcup_{\begin{pmatrix}z\\ \tilde{z}\end{pmatrix}\in\Z(n_{i-1})}\Codec\begin{pmatrix}\varphi^2(z)\\ 0^{-1}\varphi^2(\tilde{z})0\end{pmatrix}\,.
\end{equation}
For the use of~\eqref{Z NMn} we need that $\varphi^2(\tilde{z})$ has the prefix $0$ for all $\begin{pmatrix}z\\ \tilde{z}\end{pmatrix}\in\Z(n_{i-1})$, in other words, that $\varphi^2(\tilde{z})$ begins with $0$. With regard to~\eqref{Tribonacci}, this condition is trivially satisfied for any $\tilde{z}$.

Let $i=1$. Then
$$
\Z(n_1)=\bigcup_{\begin{pmatrix}z\\ \tilde{z}\end{pmatrix}\in\left\{\begin{pmatrix}01\\10\end{pmatrix},\begin{pmatrix}02\\20\end{pmatrix},\begin{pmatrix}0\\0\end{pmatrix}\right\}}\Codec\begin{pmatrix}\varphi^2(z)\\ 0^{-1}\varphi^2(\tilde{z})0\end{pmatrix}\,.
$$
We have
$$
\begin{array}{ccccccccccc}
\varphi^2(01)&=&01\:0\:201\:0 && \varphi^2(02)&=&01\:0201 && \varphi^2(0)&=&01\:02 \\
0^{-1}\varphi^2(10)0&=&10\:0\:102\:0 && 0^{-1}\varphi^2(20)0&=&10\:1020 && 0^{-1}\varphi^2(0)0&=&10\:20
\end{array}
$$
hence
$$
\Z(n_1)=
\left\{\begin{pmatrix}01\\10\end{pmatrix},\begin{pmatrix}0\\0\end{pmatrix},\begin{pmatrix}201\\102\end{pmatrix},\begin{pmatrix}0201\\1020\end{pmatrix},\begin{pmatrix}02\\20\end{pmatrix}\right\}\,.
$$

Let $i=2$. Then
$$
\Z(n_2)=
\bigcup_{\begin{pmatrix}z\\ \tilde{z}\end{pmatrix}\in\Z(n_{1})}\Codec\begin{pmatrix}\varphi^2(z)\\ 0^{-1}\varphi^2(\tilde{z})0\end{pmatrix}\,,
$$
and since $Z(n_{1})=\Z(n_{0})\cup\left\{\begin{pmatrix}201\\102\end{pmatrix},\begin{pmatrix}0201\\1020\end{pmatrix}\right\}$, it holds
\begin{multline*}
\Z(n_2)=\bigcup_{\begin{pmatrix}z\\ \tilde{z}\end{pmatrix}\in\Z(n_{0})}\Codec\begin{pmatrix}\varphi^2(z)\\ 0^{-1}\varphi^2(\tilde{z})0\end{pmatrix}\cup\Codec\begin{pmatrix}\varphi^2(201)\\ 0^{-1}\varphi^2(102)0\end{pmatrix}\cup\Codec\begin{pmatrix}\varphi^2(0201)\\ 0^{-1}\varphi^2(1020)0\end{pmatrix}=\\
=\Z(n_1)\cup\Codec\begin{pmatrix}\varphi^2(201)\\ 0^{-1}\varphi^2(102)0\end{pmatrix}\cup\Codec\begin{pmatrix}\varphi^2(0201)\\ 0^{-1}\varphi^2(1020)0\end{pmatrix}\,.
\end{multline*}
We have
$$
\begin{array}{ccccccc}
\varphi^2(201)&=&01\:0\,1\,0\,2\,0\,1\,0 && \varphi^2(0201)&=&01\:0\:201\:0\,1\,0\:2010 \\
0^{-1}\varphi^2(102)0&=&10\:0\,1\,0\,2\,0\,1\,0 && 0^{-1}\varphi^2(1020)0&=&10\:0\:102\:0\,1\,0\:1020
\end{array}
$$
therefore abelian co-decompositions of $\begin{pmatrix}\varphi^2(201)\\ 0^{-1}\varphi^2(102)0\end{pmatrix}$ and $\begin{pmatrix}\varphi^2(0201)\\ 0^{-1}\varphi^2(1020)0\end{pmatrix}$ are
$$
\Codec\begin{pmatrix}\varphi^2(201)\\ 0^{-1}\varphi^2(102)0\end{pmatrix}=\left\{\begin{pmatrix}01\\10\end{pmatrix},\begin{pmatrix}0\\0\end{pmatrix},\begin{pmatrix}1\\1\end{pmatrix},\begin{pmatrix}2\\2\end{pmatrix}\right\}\,,
$$
$$
\Codec\begin{pmatrix}\varphi^2(0201)\\ 0^{-1}\varphi^2(1020)0\end{pmatrix}=\left\{\begin{pmatrix}01\\10\end{pmatrix},\begin{pmatrix}0\\0\end{pmatrix},\begin{pmatrix}201\\102\end{pmatrix},\begin{pmatrix}1\\1\end{pmatrix}\right\}\,,
$$
hence
$$
\Z(n_2)=\Z(n_1)\cup\left\{\begin{pmatrix}01\\10\end{pmatrix},\begin{pmatrix}0\\0\end{pmatrix},\begin{pmatrix}1\\1\end{pmatrix},\begin{pmatrix}2\\2\end{pmatrix},\begin{pmatrix}201\\102\end{pmatrix}\right\}=\Z(n_1)\cup\left\{\begin{pmatrix}1\\1\end{pmatrix},\begin{pmatrix}2\\2\end{pmatrix}\right\}\,.
$$

Let $i=3$. Then
$$
\Z(n_3)=
\bigcup_{\begin{pmatrix}z\\ \tilde{z}\end{pmatrix}\in\Z(n_{2})}\Codec\begin{pmatrix}\varphi^2(z)\\ 0^{-1}\varphi^2(\tilde{z})0\end{pmatrix}\,,
$$
and since $\Z(n_{2})=\Z(n_{1})\cup\left\{\begin{pmatrix}1\\1\end{pmatrix},\begin{pmatrix}2\\2\end{pmatrix}\right\}$, it holds
\begin{multline*}
\Z(n_3)=\bigcup_{\begin{pmatrix}z\\ \tilde{z}\end{pmatrix}\in\Z(n_{1})}\Codec\begin{pmatrix}\varphi^2(z)\\ 0^{-1}\varphi^2(\tilde{z})0\end{pmatrix}\cup\Codec\begin{pmatrix}\varphi^2(1)\\ 0^{-1}\varphi^2(1)0\end{pmatrix}\cup\Codec\begin{pmatrix}\varphi^2(2)\\ 0^{-1}\varphi^2(2)0\end{pmatrix}=\\
=\Z(n_2)\cup\Codec\begin{pmatrix}\varphi^2(1)\\ 0^{-1}\varphi^2(1)0\end{pmatrix}\cup\Codec\begin{pmatrix}\varphi^2(2)\\ 0^{-1}\varphi^2(2)0\end{pmatrix}\,.
\end{multline*}
We have
$$
\Codec\begin{pmatrix}\varphi^2(1)\\ 0^{-1}\varphi^2(1)0\end{pmatrix}\cup\Codec\begin{pmatrix}\varphi^2(2)\\ 0^{-1}\varphi^2(2)0\end{pmatrix}=\Codec\begin{pmatrix}010\\ 100\end{pmatrix}\cup\Codec\begin{pmatrix}01\\ 10\end{pmatrix}=
\left\{\begin{pmatrix}01\\10\end{pmatrix},\begin{pmatrix}0\\0\end{pmatrix}\right\}\,,
$$
hence
$$
\Z(n_3)=\Z(n_2)\,.
$$
Then~\eqref{Uk.4} implies $\Z(n_4)=\Z(n_3)$, $\Z(n_5)=\Z(n_4)$ etc.,
therefore $\Z(n_i)=\Z(n_2)$ for all $i\geq2$.
\end{proof}

\begin{corollary}\label{Coro.Trib.4}
If $\langle n_i\rangle_F=((1,0)^i,1)$, then
$\AC(n_i)=4$ for all $i\in\N$. Consequently, the abelian complexity of the Tribonacci word attains the value $4$ infinitely many times.
\end{corollary}

\begin{proof}
We apply formula~\eqref{Pred(n).} on the results of Proposition~\ref{Trib.4}. It holds
$$
\Pred(n_1)=\bigcup_{\begin{pmatrix}z\\ \tilde{z}\end{pmatrix}\in\Z(n_1)}\left\{\Psi(s)-\Psi(r)\;\left|\;\text{$r$ is a prefix of $x$, $s$ is a prefix of $\tilde{z}$, $|s|=|r|$}\right.\right\}\,,
$$
and since $\Z(n_1)=\left\{\begin{pmatrix}01\\10\end{pmatrix},\begin{pmatrix}0\\0\end{pmatrix},\begin{pmatrix}201\\102\end{pmatrix},\begin{pmatrix}0201\\1020\end{pmatrix},\begin{pmatrix}02\\20\end{pmatrix}\right\}$, we have
\begin{multline*}
\Pred(n_1)=\left\{(-1,1,0),(0,0,0)\right\}\cup\left\{(0,0,0)\right\}\cup\left\{(0,1,-1),(0,0,0)\right\}\cup\\
\cup\left\{(-1,1,0),(0,1-1),(0,0,0)\right\}\cup\left\{,(-1,0,1),(0,0,0)\right\}=\\
=\left\{(-1,1,0),(0,0,0),(0,1,-1),(-1,0,1)\right\}\,.
\end{multline*}
For $i\geq2$, it is easy to see that $\Pred(n_i)=\Pred(n_1)$.

Then $\AC(n_i)=\#\Pred(n_i)=4$ for all $i\in\N$, see formula~\eqref{ACred}.
\end{proof}


\subsection{$\AC$ attains the value $5$ infinitely many times}\label{AC5}

\begin{proposition}\label{Trib.5}
If $\langle N_i\rangle_F=(\underbrace{1,0,0,0,\ldots,1,0,0,0}_{\text{$i$ times $(1,0,0,0)$}},1)=((1,0,0,0)^i,1)$, then
\begin{align*}
&\Z(N_2) =  \left\{\begin{pmatrix}01\\10\end{pmatrix},\begin{pmatrix}02\\20\end{pmatrix},\begin{pmatrix}0\\0\end{pmatrix}\begin{pmatrix}0102\\2010\end{pmatrix},\begin{pmatrix}1\\1\end{pmatrix},\begin{pmatrix}2\\2\end{pmatrix},\begin{pmatrix}201\\102\end{pmatrix}\right\}, \\
&\Z(N_i)  =  \left\{\begin{pmatrix}01\\10\end{pmatrix},\begin{pmatrix}02\\20\end{pmatrix},\begin{pmatrix}0\\0\end{pmatrix}\begin{pmatrix}0102\\2010\end{pmatrix},\begin{pmatrix}1\\1\end{pmatrix},\begin{pmatrix}2\\2\end{pmatrix},\begin{pmatrix}201\\102\end{pmatrix},\begin{pmatrix}00102\\20100\end{pmatrix}\right\} \;\text{for all $i\geq3$}.
\end{align*}
\end{proposition}

\begin{proof}
We proceed in the same way as in the proof of Proposition~\ref{Trib.4}.

We have $\langle N_0\rangle_F=(1)$; $\Z(1)$ is known from the proof of Proposition~\ref{Trib.4}:
$$
\Z(N_0)=\left\{\begin{pmatrix}01\\10\end{pmatrix},\begin{pmatrix}02\\20\end{pmatrix},\begin{pmatrix}0\\0\end{pmatrix}\right\}\,.
$$

For all $i\geq1$, it holds $\langle N_i\rangle_F=(\langle N_{i-1}\rangle_F,0,0,0,1)$, therefore $\Z(N_i)$ will be derived from $\Z(N_{i-1})$ using formula~\eqref{Z NMn} to which we substitute $\langle q\rangle_F=(0,0,0,1)$, $Q=N_{i-1}$.

It is required that $\varphi^4(\tilde{z})$ has the prefix $\mathbf{t}_{[q]}=0$ for all $\begin{pmatrix}z\\ \tilde{z}\end{pmatrix}\in\Z(N_{i-1})$, which is obviously satisfied for any $\tilde{z}$.

Let us begin with $i=1$.
$$
\Z(N_1)=\bigcup_{\begin{pmatrix}z\\ \tilde{z}\end{pmatrix}\in\left\{\begin{pmatrix}01\\10\end{pmatrix},\begin{pmatrix}02\\20\end{pmatrix},\begin{pmatrix}0\\0\end{pmatrix}\right\}}\Codec\begin{pmatrix}\varphi^4(z)\\ 0^{-1}\varphi^4(\tilde{z})0\end{pmatrix}\,.
$$
We have
$$
\begin{array}{ccccccc}
\varphi^4(01)&=&01\:02\:01\:0\:01\:02\:01\:0102\:0\:1\:0\:0102 \\
0^{-1}\varphi^4(10)0&=&10\:20\:10\:0\:10\:20\:10\:2010\:0\:1\:0\:2010
\end{array}
$$
$$
\begin{array}{ccccccc}
\varphi^4(02)&=&01\:02\:01\:0\:01\:0201\:0\:1\:0\:2\:0\:1\:0 && \varphi^4(0)&=&01\:02\:01\:0\:01\:0201 \\
0^{-1}\varphi^4(20)f&=&10\:20\:10\:0\:10\:2010\:0\:1\:0\:2\:0\:1\:0 && 0^{-1}\varphi^4(0)0&=&10\:20\:10\:0\:10\:2010
\end{array}
$$
hence
$$
\Z(N_1)=\left\{\begin{pmatrix}01\\10\end{pmatrix},\begin{pmatrix}02\\20\end{pmatrix},\begin{pmatrix}0\\0\end{pmatrix},\begin{pmatrix}0102\\2010\end{pmatrix},\begin{pmatrix}1\\1\end{pmatrix},\begin{pmatrix}2\\2\end{pmatrix}\right\}=\Z(N_0)\cup\left\{\begin{pmatrix}0102\\2010\end{pmatrix},\begin{pmatrix}1\\1\end{pmatrix},\begin{pmatrix}2\\2\end{pmatrix}\right\}\,.
$$

We proceed in the same way for other values of $k$:

If $i=2$, we obtain
\begin{multline*}
\Z(N_2)
=\bigcup_{\begin{pmatrix}z\\ \tilde{z}\end{pmatrix}\in\Z(N_{0})\cup\left\{\begin{pmatrix}0102\\2010\end{pmatrix},\begin{pmatrix}1\\1\end{pmatrix},\begin{pmatrix}2\\2\end{pmatrix}\right\}}\Codec\begin{pmatrix}\varphi^4(z)\\ 0^{-1}\varphi^4(\tilde{z})0\end{pmatrix}=\\
=\Z(N_1)\cup\Codec\begin{pmatrix}\varphi^4(0102)\\ 0^{-1}\varphi^4(2010)0\end{pmatrix}\cup\Codec\begin{pmatrix}\varphi^4(1)\\ 0^{-1}\varphi^4(1)0\end{pmatrix}\cup\Codec\begin{pmatrix}\varphi^4(2)\\ 0^{-1}\varphi^4(2)0\end{pmatrix}\,.
\end{multline*}
Since it holds
$$
\begin{array}{ccccccccccccccccccc}
\varphi^4(0102)&=&01\:02\:01\:0\:01\:02\:01\:0102010\:01\:02\:01\:0\:201\:0\:01\:02\:01\:0102010 \\
0^{-1}\varphi^4(2010)0&=&10\:20\:10\:0\:10\:20\:10\:0102010\:10\:20\:10\:0\:102\:0\:10\:20\:10\:0102010
\end{array}
$$
$$
\begin{array}{ccccccccccc}
\varphi^4(1)&=&01\:02\:01\:0\:01\:02 && \varphi^4(2)&=&01\:02\:01\:0 \\
0^{-1}\varphi^4(1)0&=&10\:20\:10\:0\:10\:20 && 0^{-1}\varphi^4(2)0&=&10\:20\:10\:0
\end{array}
$$
we have
$$
\Z(N_2)=\Z(N_1)\cup\left\{\begin{pmatrix}01\\10\end{pmatrix},\begin{pmatrix}02\\20\end{pmatrix},\begin{pmatrix}0\\0\end{pmatrix},\begin{pmatrix}1\\1\end{pmatrix},\begin{pmatrix}2\\2\end{pmatrix},\begin{pmatrix}201\\102\end{pmatrix}\right\}=\Z(N_1)\cup\left\{\begin{pmatrix}201\\102\end{pmatrix}\right\}\,.
$$

Let $i=3$. Then
$$
\Z(N_3)
=\bigcup_{\begin{pmatrix}z\\ \tilde{z}\end{pmatrix}\in\Z(N_1)\cup\left\{\begin{pmatrix}201\\102\end{pmatrix}\right\}}\Codec\begin{pmatrix}\varphi^4(z)\\ 0^{-1}\varphi^4(\tilde{z})0\end{pmatrix}
=\Z(N_2)\cup\Codec\begin{pmatrix}\varphi^4(201)\\ 0^{-1}\varphi^4(102)0\end{pmatrix}\,.
$$
It holds
$$
\begin{array}{ccccccccccccccccccc}
\varphi^4(201)&=&01\:02\:01\:0\:01\:02\:01\:00102\:01\:0102\:01\:00102 \\
0^{-1}\varphi^4(102)0&=&10\:20\:10\:0\:10\:20\:10\:20100\:10\:2010\:10\:20100
\end{array}
$$
hence
$$
\Z(N_3)=\Z(N_2)\cup\left\{\begin{pmatrix}01\\10\end{pmatrix},\begin{pmatrix}02\\20\end{pmatrix},\begin{pmatrix}0\\0\end{pmatrix},\begin{pmatrix}00102\\20100\end{pmatrix},\begin{pmatrix}0102\\2010\end{pmatrix}\right\}=\Z(N_2)\cup\left\{\begin{pmatrix}00102\\20100\end{pmatrix}\right\}\,.
$$

Let $i=4$. Then
\begin{multline*}
\Z(N_4)
=\bigcup_{\begin{pmatrix}z\\ \tilde{z}\end{pmatrix}\in\Z(N_2)\cup\left\{\begin{pmatrix}00102\\20100\end{pmatrix}\right\}}\Codec\begin{pmatrix}\varphi^4(z)\\ 0^{-1}\varphi^4(\tilde{z})0\end{pmatrix}=\\
=\Z(N_3)\cup\Codec\begin{pmatrix}\varphi^4(00102)\\ 0^{-1}\varphi^4(20100)0\end{pmatrix}\,.
\end{multline*}
It holds
$$
\begin{array}{ccccccccccccccccccc}
\varphi^4(00102)&=&01\:02\:01\:0\:01\:02\:01\:0102010\:01\:02\:01\:0102010\:0102\:0102010\:01\:02\:01\:0102010 \\
0^{-1}\varphi^4(20100)0&=&10\:20\:10\:0\:10\:20\:10\:0102010\:10\:20\:10\:0102010\:2010\:0102010\:10\:20\:10\:0102010
\end{array}
$$
hence
$$
\Z(N_4)=\Z(N_3)\cup\left\{\begin{pmatrix}01\\10\end{pmatrix},\begin{pmatrix}02\\20\end{pmatrix},\begin{pmatrix}0\\0\end{pmatrix},\begin{pmatrix}0102\\2010\end{pmatrix}\right\}\,,
$$
which means $\Z(N_4)=\Z(N_3)$, and consequently, for all $i\geq4$, 
$$
\Z(N_i)=\Z(N_{i-1})=\Z(N_3)\,.
$$
\end{proof}

\begin{corollary}
If $\langle N_i\rangle_F=((1,0,0,0)^i,1)$, then
$$
\Pred(N_i)=\left\{(0,0,0),(-1,1,0),(-1,0,1),(0,-1,1),(0,1,-1)\right\} \quad\text{for all $i\geq2$}\,,
$$
therefore 
$\AC(N_i)=5$ for all $i\geq2$. Consequently, the abelian complexity of the Tribonacci word attains the value $5$ infinitely many times.
\end{corollary}

\begin{proof}
The statement can be proved using formula~\eqref{Pred(n).} in the same way as Corollary~\ref{Coro.Trib.4}.
\end{proof}


\subsection{$\AC$ attains the value $6$ infinitely many times}

\begin{proposition}\label{Trib.6}
If $\langle M_i\rangle_F=(\underbrace{1,0,0,0,\ldots,1,0,0,0}_{\text{$i$ times $(1,0,0,0)$}},0)=((1,0,0,0)^i,0)$, then
$$
\Z(M_i) =  \left\{\begin{pmatrix}0\\0\end{pmatrix},\begin{pmatrix}1\\1\end{pmatrix},\begin{pmatrix}2\\2\end{pmatrix},\begin{pmatrix}102\\201\end{pmatrix},\begin{pmatrix}10\\01\end{pmatrix},\begin{pmatrix}20\\02\end{pmatrix},\begin{pmatrix}2010\\0102\end{pmatrix},\begin{pmatrix}0102\\2010\end{pmatrix}\right\}
$$
for all $i\geq3$.
\end{proposition}

\begin{proof}

For all $i\geq1$, it holds $\langle M_i\rangle_F=(\langle N_{i-1}\rangle_F,0,0,0,0)$, where $N_i$ are the numbers defined in Proposition~\ref{Trib.5}. Therefore $\Z(M_i)$ can be determined using the result of Proposition~\ref{Trib.5} and formula~\eqref{Z M0}:
$$
\Z(M_i)=\left\{\left.\Codec\begin{pmatrix}\varphi^4(z)\\ \varphi^4(\tilde{z})\end{pmatrix}\;\right|\;\begin{pmatrix}z\\ \tilde{z}\end{pmatrix}\in\Z(N_{i-1})\right\}\,.
$$

For $i=3$, we have
$
\Z(M_3)=
\bigcup_{\begin{pmatrix}z\\ \tilde{z}\end{pmatrix}\in\Z(N_2)}\Codec\begin{pmatrix}\varphi^4(z)\\ \varphi^4(\tilde{z})\end{pmatrix}
$,
where
$$
\Z(N_2)=\left\{\begin{pmatrix}01\\10\end{pmatrix},\begin{pmatrix}02\\20\end{pmatrix},\begin{pmatrix}0\\0\end{pmatrix},\begin{pmatrix}0102\\2010\end{pmatrix},\begin{pmatrix}1\\1\end{pmatrix},\begin{pmatrix}2\\2\end{pmatrix},\begin{pmatrix}201\\102\end{pmatrix}\right\}
$$
see Section~\ref{AC5}. Let us find the decompositions:
$$
\begin{array}{ccccccccccccccccccc}
\varphi^4(01)&=&01020100102010\:102\:0\:10\:0\:102 && \varphi^4(0)&=&0102010010201 \\
\varphi^4(10)&=&01020100102010\:201\:0\:01\:0\:201 && \varphi^4(0)&=&0102010010201
\end{array}
$$
$$
\begin{array}{ccccccccccc}
\varphi^4(02)&=&01020100102010\:10\:20\:10 && \varphi^4(1)&=&01020100102 \\
\varphi^4(20)&=&01020100102010\:01\:02\:01 && \varphi^4(1)&=&01020100102
\end{array}
$$
$$
\begin{array}{ccccccccccccccccccc}
\varphi^4(0102)&=&01020100102010\:10\:20\:10\:0102010\:2010\:0102010\:10\:20\:10 \\
\varphi^4(2010)&=&01020100102010\:01\:02\:01\:0102010\:0102\:0102010\:01\:02\:01
\end{array}
$$
$$
\begin{array}{ccccccccccc}
\varphi^4(2)&=&0102010 && \varphi^4(201)&=&01020100102010\:0102\:010\:102\:010\:0102\\
\varphi^4(2)&=&0102010 && \varphi^4(102)&=&01020100102010\:2010\:010\:201\:010\:2010
\end{array}
$$
Hence
$$
\Z(M_3)=\left\{\begin{pmatrix}0\\0\end{pmatrix},\begin{pmatrix}1\\1\end{pmatrix},\begin{pmatrix}2\\2\end{pmatrix},\begin{pmatrix}102\\201\end{pmatrix},\begin{pmatrix}10\\01\end{pmatrix},\begin{pmatrix}20\\02\end{pmatrix},\begin{pmatrix}2010\\0102\end{pmatrix},\begin{pmatrix}0102\\2010\end{pmatrix}\right\}\,.
$$

Let us proceed to $i\geq4$. It holds $\Z(N_{i-1})=\Z(N_3)$ for all $i\geq 4$, therefore
\begin{multline*}
\Z(M_i)
=\bigcup_{\begin{pmatrix}z\\ \tilde{z}\end{pmatrix}\in\Z(N_{i-1})}\Codec\begin{pmatrix}\varphi^4(z)\\ \varphi^4(\tilde{z})\end{pmatrix}=
\bigcup_{\begin{pmatrix}z\\ \tilde{z}\end{pmatrix}\in\Z(N_3)}\Codec\begin{pmatrix}\varphi^4(z)\\ \varphi^4(\tilde{z})\end{pmatrix}=\\
=
\bigcup_{\begin{pmatrix}z\\ \tilde{z}\end{pmatrix}\in\Z(N_2)\cup\left\{\begin{pmatrix}00102\\20100\end{pmatrix}\right\}}\Codec\begin{pmatrix}\varphi^4(z)\\ \varphi^4(\tilde{z})\end{pmatrix}
=\Z(M_3)\cup\Codec\begin{pmatrix}\varphi^4(00102)\\ \varphi^4(20100)\end{pmatrix}\,.
\end{multline*}
Since
$$
\begin{array}{ccccccccccccccccccc}
\varphi^4(00102)&=&01020100102010\:10\:20\:10\:0102010\:10\:20\:10\:0\:102\:0\:10\:20\:10\:0102010\:10\:20\:10 \\
\varphi^4(20100)&=&01020100102010\:01\:02\:01\:0102010\:01\:02\:01\:0\:201\:0\:01\:02\:01\:0102010\:01\:02\:01
\end{array}
$$
we obtain for all $i\geq4$:
\begin{multline*}
\Z(M_i)=\Z(M_3)\cup\left\{\begin{pmatrix}0\\0\end{pmatrix},\begin{pmatrix}1\\1\end{pmatrix},\begin{pmatrix}2\\2\end{pmatrix},\begin{pmatrix}10\\01\end{pmatrix},\begin{pmatrix}20\\02\end{pmatrix},\begin{pmatrix}102\\201\end{pmatrix}\right\}=\\
=\left\{\begin{pmatrix}0\\0\end{pmatrix},\begin{pmatrix}1\\1\end{pmatrix},\begin{pmatrix}2\\2\end{pmatrix},\begin{pmatrix}102\\201\end{pmatrix},\begin{pmatrix}10\\01\end{pmatrix},\begin{pmatrix}20\\02\end{pmatrix},\begin{pmatrix}2010\\0102\end{pmatrix},\begin{pmatrix}0102\\2010\end{pmatrix}\right\}\,.
\end{multline*}
\end{proof}

\begin{corollary}
If $\langle M_i\rangle_F=((1,0,0,0)^i,0)$, then for $i\geq3$,
$$
\Pred(M_i)=\left\{(0,0,0),(0,-1,1),(1,-1,0),(1,0,-1),(0,1,-1),(-1,0,1)\right\},
$$
therefore $\AC(M_i)=6$ for all $i\geq3$. Consequently, the abelian complexity of the Tribonacci word attains the value $6$ infinitely many times.
\end{corollary}

\begin{proof}
The statement can be proved using formula~\eqref{Pred(n).} in the same way as Corollary~\ref{Coro.Trib.4}.
\end{proof}

\begin{remark}
The reader may wonder how to find, for a given word $\u$, the sequence $(n_i)_i$ such that $\AC_\u(n_i)=k$ for all $i\geq i_0$. The simplest way is usually either to directly check ``nice'' sequences, like e.g. $\langle n_i\rangle_F=((1,0^\ell)^i,1)$ for various $\ell$, or to conduct a computer experiment and then choose $(n_i)_i$ according to its results, or to combine these two approaches.
\end{remark}

\section{Conclusions}

We have proposed a method for dealing with the abelian complexity of infinite words, based on the recurrence property of the given word and on the use of the so called relative Parikh vectors.
For the sake of simplicity and clarity, we have concentrated mainly on infinite words associated with Parry numbers, however, it is likely that other recurrent words, especially those which are fixed points of primitive substitutions, could be explored similarly.

In order to demonstrate the use of the method, we have solved an open problem related to the Tribonacci word $\mathbf{t}$. The word $\mathbf{t}$ is simple Parry, non-simple Parry words can be treated in the same way: for example, we are able to prove for the words $\u^{(p)}$ (with a parameter $p>2$) mentioned in the Introduction
that every value of $\AC_{\u^{(p)}}$ is attained infinitely many times as well. Naturally, involving a parameter makes the analysis generally more complicated, but it is still feasible.

It is possible that further development of this approach can help to solve also the more difficult problem of characterizing all numbers $n$ for which $\AC_\u(n)$ attains a specified value, at least in cases when the images of $\AC_\u$ have low cardinalities.

The application of the abelian co-decomposition to the Parry words relies on the associated normal $F$-representations.
It is noteworthy that
the procedure of determining $\AC_\u(n)$
seems to be particularly efficient if $\langle n\rangle_F$ is periodic. This fact together with the property described in Remark~\ref{infinitely} indicates that there probably exists a very close relation between $\AC_\u(n)$ and the coefficients of $\langle n\rangle_F$, in other words, if an explicit formula for $\AC_\u(n)$ is sought, it shall be sought in terms of $\langle n\rangle_F$, similarly as it has already been done for infinite words associated with quadratic Parry numbers in~\cite{BBT}.


\section*{Acknowledgements}

The author thanks Karel B\v rinda for a numerical experiment and Edita Pelantov\'a for many valuable comments on the manuscript.


\end{document}